\documentclass{amsart}

\usepackage[english]{babel}
\usepackage[utf8x]{inputenc}
\usepackage[T1]{fontenc}
\usepackage{enumitem}
\usepackage{caption}
\usepackage{subcaption}
\usepackage{comment} 

\usepackage[a4paper,top=3cm,bottom=2cm,left=3cm,right=3cm,marginparwidth=1.75cm]{geometry}

\usepackage{amsmath,amsthm,amsfonts,mathrsfs}
\usepackage{graphicx}

\newtheorem{theorem}{Theorem}
\newtheorem{lemma}[theorem]{Lemma}

\newtheorem{definition}{Definition}

\theoremstyle{definition}
\newtheorem*{example}{Example}

\newcommand{\pa}{\partial}

\title{Hyperbolicity of links in thickened surfaces}

\author{C. Adams, C. Albors-Riera, B. Haddock, Z. Li, D. Nishida, B. Reinoso and L. Wang}

\begin{document}

\begin{abstract}

Menasco showed that a non-split, prime, alternating link that is not a 2-braid is hyperbolic in $S^3$. We prove a similar result for links in closed thickened surfaces $S \times I$. We define a link to be fully alternating if it has an alternating projection from $S\times I$ to $S$ where the interior of every complementary region is an open disk. We show that a prime, fully alternating link in $S\times I$ is hyperbolic. Similar to Menasco, we also give an easy way to determine primeness in $S\times I$. A fully alternating link is prime in $S\times I$ if and only if it is ``obviously prime''. Furthermore, we extend our result to show that a prime link with fully alternating projection to an essential surface embedded in an orientable, hyperbolic 3-manifold has a hyperbolic complement. 

\end{abstract}

\maketitle

\section{Introduction}
\label{sec:intro}

A link $L$ in a manifold $M$ is hyperbolic if its complement $M\setminus L$ supports a hyperbolic metric. In~\cite{menasco}, W. Menasco showed that a non-split, prime, alternating link in $S^3$ that is not a 2-braid link is hyperbolic. In this paper, we prove an analogous result for links that are embedded in another family of manifolds, namely the closed thickened surfaces $S\times I$, such that $S$ has nonpositive Euler characteristic. Note that in this paper, we restrict our attention to a closed surface $S$, although we believe that a version of our results hold for a surface with boundary.

\begin{definition}
Let $L$ be a link in a thickened surface $S \times I$, orientable or not, with the exception of the sphere and the projective plane. A projection of $L$ from $S\times I$ to $S$ is fully alternating if it is alternating on $S$ and the interior of every complementary region is an open disk. We say a link $L$ is fully alternating in $S\times I$ if it has a fully alternating projection from $S\times I$ to $S$.
\end{definition}

A sphere $F$ in $S \times I$ punctured twice by $L$ is essential in $S \times I \setminus L$ if it does not bound a ball containing an unknotted arc of $L$.

\begin{definition}
A link $L$ is \emph{prime} in $S\times I$ if there does not exist an essential twice-punctured sphere in $S \times I \setminus L$ such that both punctures are created by $L$.
\end{definition}

\begin{theorem}
\label{thm:main}
Let $S$ be a closed orientable surface. A prime fully alternating link $L$ in $S \times I$ is hyperbolic.
\end{theorem}

In order to prove these results, we need to consider surfaces in 3-manifolds.

\begin{definition}
A surface properly embedded in a compact manifold is \emph{essential} if it is incompressible and not boundary-parallel.
\end{definition}

Note that throughout the paper, we refer to surfaces being punctured by $L$ or having boundary on $\partial{N}(L)$. In these cases we think of the surface as being embedded in $S \times I \setminus L$ or $S \times I \setminus \mathring{N}(L)$, as appropriate. We use the fact that $S \times I \setminus L$ is homeomorphic to the interior of  $S \times I \setminus \mathring{N}(L)$ as needed.
\\

Menasco further showed that a non-split, alternating link $L$ in $S^3$  is prime if and only if in any reduced alternating diagram, every circle in the projection crossing the link projection transversely twice  bounds a disk in the projection plane  containing no crossings of the projection~\cite{menasco}. We call this property ``obviously prime''. We extend this idea and determine the primeness of a link $L$ in $S\times I$, by looking at any given alternating projection of $L$ and seeing if it is obviously prime, as defined below.
\\

We define a reduced projection to be a projection that does not have any unnecessary crossings as in Figure~\ref{fig:reduced}. Note that if we start with a fully alternating projection, its reduction is also fully alternating.
\begin{figure}[h]
	\centering
	\includegraphics[scale=0.4]{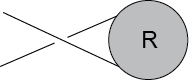}
	\caption{Crossing that can be reduced.}
	\label{fig:reduced}
\end{figure}

\begin{definition}
A reduced, fully alternating projection $P$ of a link $L$ in $S \times I$ onto $S$ is \emph{obviously prime} if every disk in the projection surface with boundary intersecting $P$ transversely at two points intersects $P$ in an embedded arc.
\end{definition}

\begin{theorem}
\label{thm:primeiffobv}
Let $L$ be a fully alternating link in $S\times I$ and let $P$ be a reduced fully alternating projection of $L$ to $S$. Then $L$ is prime if and only if $P$ is obviously prime.
\end{theorem}

To prove Theorems 1 and 2, we build on techniques first appearing in~\cite{menasco}. The results from~\cite{menasco} rely on a critical theorem by W. Thurston that a compact connected 3-manifold with torus boundaries has hyperbolic interior if and only if it does not contain any essential spheres, tori, or annuli (see~\cite{thurston}). 
\\\\
Menasco defines the notion of pairwise compressibility, which we adapt and refer to as meridional compressibility.

\begin{definition}
Let $L$ be a link in a 3-manifold $M$. Let $F$ be a surface embedded in $M\setminus L$. Then $F$ is \emph{meridionally incompressible} if for every disk $D$ with $D\cap F=\pa D$ which is punctured once by $L$, there exists another disk $D' \subset F \cup L$ with $\pa D = \pa D'$, such that $D'$ is punctured once by $L$,  and $D$ can be isotoped to $D'$ fixing its boundary. Otherwise $F$ is \emph{meridionally compressible}.
\end{definition}

In order to prove Theorem~\ref{thm:main}, we show in Section~\ref{sec:SphereTorus} that for a link $L$ in $S\times I$ that is prime and fully alternating, $S \times I \setminus L$ does not contain any incompressible, meridionally incompressible spheres, tori, or annuli with both boundary components on $\partial(S\times I)$. Further, we show that if an embedded sphere or torus is essential, then it cannot be meridionally compressible. Thus we have eliminated essential spheres, tori, and a subset of the essential annuli in $S \times I \setminus L$. We then show in Section~\ref{sec:annulus} that $S \times I \setminus L$ does not contain any other essential annuli. 
\\\\
We prove Theorem~\ref{thm:primeiffobv} in Section~\ref{sec:prime}, thereby providing a means of easily identifying primeness of links in $S\times I$ with fully alternating projections. 
\\\\
In Section~\ref{sec:extensions}, we generalize Theorem \ref{thm:main} to links in a neighborhood of an essential surface embedded in a hyperbolic manifold. 
Note that an essential surface in a hyperbolic 3-manifold must have negative Euler characteristic.

\begin{theorem}
\label{thm:threemanifold}
Let $M$ be a finite volume hyperbolic 3-manifold, possibly with cusps such that any boundary is total geodesic. Let $S$ be an essential, closed surface in $M$ with neighborhood $N$. For any link $L$ that is prime in $N$ with a fully alternating projection on $S$, the manifold $M\setminus L$ is hyperbolic.
\end{theorem}

In order to prove this we first extend Theorem~\ref{thm:main} to prime, fully alternating links in any $I$-bundle over any closed surface that is either orientable or non-orientable. We then show that the hyperbolicity of the $I$-bundle implies the hyperbolicity of $M$.
\\\\
Theorem \ref{thm:threemanifold} substantially increases the number of manifolds known to be hyperbolic. See Section~\ref{sec:extensions} for examples. The proof of the Virtual Haken Conjecture by I. Agol (see~\cite{haken}) gives an embedded incompressible surface in some finite cover $M'$ of any compact hyperbolic 3-manifold. For each such manifold $M'$, Theorem~\ref{thm:threemanifold} generates an infinite family of finite volume hyperbolic link complements. 
\\\\
A second application of Theorem~\ref{thm:main} is to tiling theory. Specifically, consider a 4-regular tiling of $\mathbb{E}^2$ or $\mathbb{H}^2$ by edge-to-edge polygons such that the symmetry group of the tiling has compact fundamental domain. By adding alternating crossings at the vertices, the tiling becomes an infinite alternating weave. By taking the quotient by an orientation-preserving subgroup of the symmetry group of the tiling, we obtain the projection of a fully alternating link on a closed orientable surface $S$ of positive genus, where the link lives in $S\times I$. This association of tilings to hyperbolic links has appeared previously for certain Euclidean uniform tilings \cite{CKP} and  for Euclidean and hyperbolic $k$-uniform tilings in   ~\cite{tiling}. By applying Theorem~\ref{thm:main}, we can now turn a broader class of tilings, with polygons not necessarily regular, into hyperbolic links in $S\times I$. 
\\

If the link in $S\times I$ associated to a tiling is hyperbolic, then we can calculate the volume of its complement. We can then assign to the infinite tiling a volume density, given by the volume of the embedded link in $S \times I$ divided by the number of crossings of the projection of $L$ to $S$ in the fundamental domain. Thus we can apply hyperbolic invariants coming out of hyperbolic manifold theory to tilings. Note that by adding bigon faces to 3-regular tilings, we can turn them into 4-regular tilings to obtain similar results. See~\cite{tiling} for more details and explicit calculations.
\\

In recent work using different methods(cf. \cite{HP}), Howie and Purcell have proved a theorem that is more general than our Theorem \ref{thm:threemanifold}, in the case that the manifold and surface are orientable. They do not require the surface to be incompressible, but rather have a weaker condition on representativity in the manifold. They also investigate the volumes and geometries of the resulting manifolds. See also \cite{CKP} where hyperbolicity is proved for appropriate alternating links in a thickened torus, which is a case considered in our Theorem \ref{thm:main}. Note that in these two papers, our term "obviously prime" corresponds to what they define as "weakly prime".

\section{Eliminating essential, meridionally incompressible surfaces}
\label{sec:SphereTorus}

Let $L$ be a fully alternating link in $S \times I$, where $S$ is a closed orientable surface of genus $g\geq 1$. Let $F$ be an essential, meridionally incompressible surface embedded in $S \times I \setminus L$. We consider one of the following cases:
\begin{enumerate}
\item $F$ is a sphere.

\item There are no essential spheres in $S \times I \setminus L$, and $F$ is a torus.

\item There are no essential spheres in $S \times I \setminus L$, and $F$ is an annulus whose boundary components both lie on $\pa (S \times I)$.
\end{enumerate}

Let $S_0=S\times \{\frac{1}{2}\}$. Consider a fully alternating projection of $L$ onto $S_0$. As in~\cite{menasco} we place a ball at each crossing, which we hereafter refer to as a bubble $B$. We note that as $L$ is prime and fully alternating, the projection is connected.
\\

Let the overstrand of $L$ at each crossing run over the top of the bubble, and the understrand run under the bottom, as depicted in Figure~\ref{fig:bubble}. In particular, both the overstrand and the understrand are in $\partial B$. We then define $S_+$ to be $S_0$ where the equatorial disk in each bubble is replaced by the upper hemisphere, denoted $\pa B_+$. Similarly, define $S_-$ to be $S_0$ where the equatorial disk in each bubble is replaced by the lower hemisphere, denoted $\pa B_-$.
\\

\begin{figure}[h]
	\centering
	\includegraphics[scale=0.2]{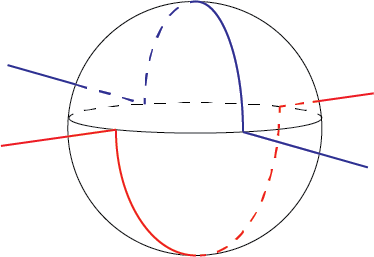}
	\caption{A bubble.}
	\label{fig:bubble}
\end{figure}

We clean up our surface $F$ relative to the bubbles by pushing $F$ radially away from the central vertical axis of each bubble $B$, as in~\cite{menasco}. As $F$ lives in the complement of $L$, we observe that what remains of $F$ forms saddles inside $B$, the boundaries of which lie on $\partial B$ and avoid the two arcs in $L\cap\partial B$. 
\\

We are interested in the intersection curves in $F\cap S_+$ and $F \cap S_-$. Hereafter we refer to $S_+$ and $F \cap S_+$, but consider all arguments and constructions as applied both to $F\cap S_+$ and $F \cap S_-$.
\begin{lemma}
\label{lem:trivialIffTrivial}
An intersection curve is trivial on $F$ if and only if it is also trivial on $S_+$.
\end{lemma}

\begin{proof}
Suppose an intersection curve $\alpha$ is trivial on $F$ and nontrivial on $S_+$. We fill the link in to work in $S\times I$. Note that $S_+$, which is incompressible in $S\times I\setminus L$, is also incompressible in $S\times I$. Consider the set of intersection curves contained in the disk bounded by $\alpha$ on $F$ that are nontrivial on $S_+$. Take one that is innermost on $F$; call it $\beta$. Then all intersection curves contained in the disk bounded by $\beta$ on $F$ must be trivial on $S_+$. Of all such trivial intersection curves contained in the disk bounded by $\beta$ on $F$, let $\gamma$ be the one that is innermost on $F$. Then the disk bounded by $\gamma$ on $F$ does not contain any other intersection curves, so we can isotope $F$ to remove this intersection curve. We can iterate this process for all trivial intersection curves contained in the disk bounded by $\beta$ on $F$. Then $\beta$ would bound a compression disk of $S_+$ on $F$, a contradiction.
\\\\
Now suppose an intersection curve $\alpha'$ is trivial on $S_+$. In case (1), $\alpha'$ is clearly trivial on $F$. Consider cases (2) and (3), and note that by hypothesis $S \times I \setminus L$ does not contain any essential spheres. Assume for contradiction that $\alpha'$ is nontrivial on $F$. Consider the set of intersection curves contained in the disk bounded by $\alpha'$ on $S_+$ that are nontrivial on $F$. Take the one that is innermost on $S_+$; call it $\beta'$. Then all intersection curves contained in the disk bounded by $\beta'$ on $S_+$ must be trivial on $F$. Of all such trivial intersection curves contained in the disk bounded by $\beta'$ on $S_+$, let $\gamma'$ be the one that is innermost on $S_+$. We can consider the disks $D_1$ and $D_2$ bounded by $\gamma'$ on $F$ and $S_+$ respectively. Identify $D_1$ and $D_2$ along $\gamma'$ to get a sphere in $S \times I \setminus L$. Given that the boundaries $\pa (S \times I)$ are outside the sphere, and that there are no essential spheres in this case, this sphere must bound a ball. Therefore we can isotope the disk $D_1$ to the disk $D_2$, and push it slightly past $S_+$ to remove $\gamma'$. We can iterate this process for all trivial intersection curves contained in the disk bounded by $\beta'$ on $S_+$. Then $\beta'$ would bound a compression disk of $F$ on $S_+$, a contradiction.
\end{proof}

Thus we say that a component of $F\cap S_+$ is trivial if it is trivial on either $F$ or $S_+$. 
\\\\
We associate an ordered pair $(s,i)$ to each embedding of $F$ prior to isotopy, in which $s$ is the number of saddles in $F$ and $i$ is the number of intersection curves in $F\cap S_+$. Pick $F$ to be the embedding such that its ordered pair $(s,i)$ is minimal under lexicographical ordering. Note that as $F$ passes through a bubble $B$, the saddle corresponds to two intersection curves on $S_+$ that run parallel to the overstrand of $B$. We think of the overstrand as dividing $\pa B$ into two sides, and are interested in which side an intersection curve hits.

\begin{lemma}
\label{lem:isotopyBubbleCleanUp}
There exists an isotopy of $F$ such that the following are true:
\begin{enumerate}
\item[(i)] The set of intersection curves $F \cap S_+$ is nonempty.
\item[(ii)] Every intersection curve in $F \cap S_+$ intersects at least one bubble.
\item[(iii)] Let $\alpha$ be an arc of an intersection curve in $F \cap S_+$ that begins and ends on the same side of a bubble $B$. Then $\alpha \cup \pa B_+$ must contain a nontrivial simple closed curve on $S_+$.

\end{enumerate}
\end{lemma}

\begin{proof}
\text{}
\begin{enumerate}[label=(\roman*), wide]
\item[(i)] We know that $F\cap S_+$ is nonempty, for otherwise $F$ would be either boundary-parallel or compressible.
\\
\item[(iii)] 
Note that an intersection arc $\alpha$ that satisfies the hypotheses of condition (iii)  must correspond to two saddles. If there is an additional intersection arc corresponding to a pair of saddles between these two, we consider the innermost pair of saddles. Using the technique in~\cite{adams toroidal} as in Figure~\ref{fig:bubbleSameSideTwice}, we isotope $F$ by pulling a neighborhood of an arc on $F$ to the bubble to form a band connecting the pair of saddles. If $\alpha \cup \pa B_+$ does not contain a simple closed curve that is nontrivial on $S_+$, we can pull the two saddles and the band through the bubble and out the other side. Note that we have decreased the number of saddles in $F\cap S_+$, contradicting that $F\cap S_+$ has the minimal number of saddles.

\begin{figure}[h]
\centering
\includegraphics[width=4cm, height=5cm]{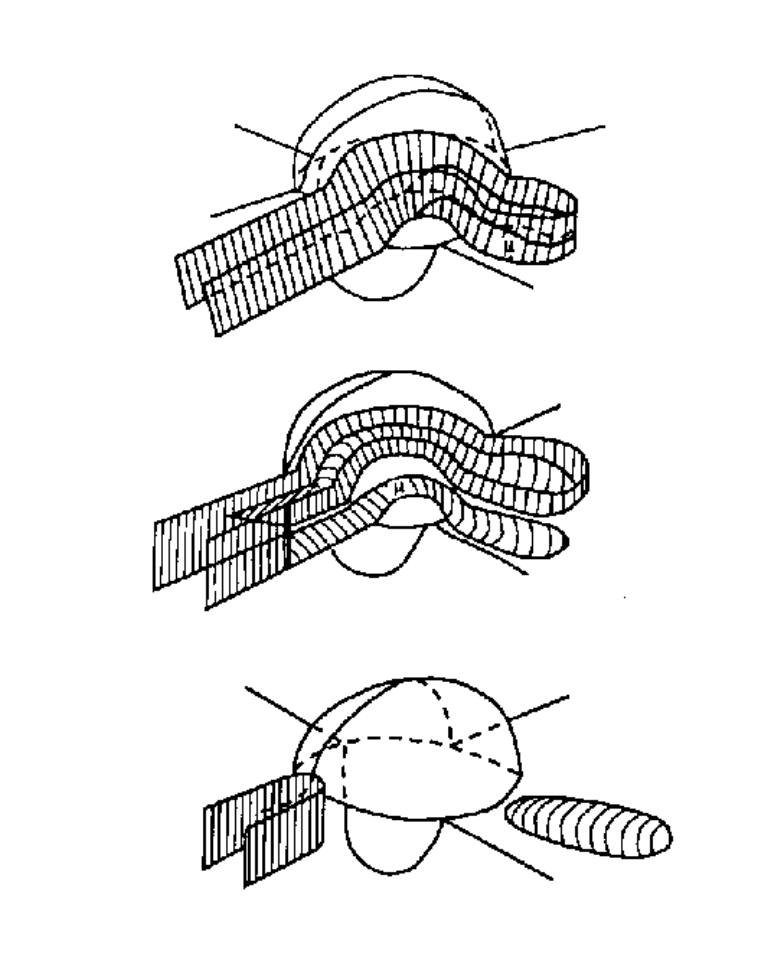}
\caption{Eliminate curve crossing a bubble twice on the same side~\cite{adams toroidal}.}
\label{fig:bubbleSameSideTwice}
\end{figure}

\item[(ii)] We additionally show that there are no intersection curves that do not hit any bubbles. Assume for contradiction that such a curve $\alpha$ exists. 
\\\\
Suppose $\alpha$ is trivial. Because the projection of $L$ is fully alternating, it is connected. Since $\alpha$ does not pass through any bubbles, there can be no link components contained in the disk bounded by $\alpha$ on $S_+$. Of the set of intersection curves contained in the disk bounded by $\alpha$ on $S_+$, choose the one that is innermost; call it $\beta$. Take the union of the disks $D_1$ and $D_2$ bounded by $\beta$ on $F$ and $S_+$ respectively to obtain a sphere in $S \times I \setminus L$, which must bound a ball. Therefore we can isotope $D_1$ to $D_2$, and push it slightly past $S_+$ to remove $\beta$. We have thus reduced the number of intersection curves without affecting the number of saddles, contradicting that $F$ has the least number of intersection curves among all isotopies with the same number of saddles.
\\\\
Now suppose $\alpha$ is nontrivial. Then, since the projection of $L$ onto $S_0$ is fully alternating, and as a nontrivial intersection curve on $S_0$ does not bound a disk, $\alpha$ must pass through at least two disk regions on $S_0$ and therefore pass through a bubble. This shows that $F$ must intersect at least one bubble.
\end{enumerate}
\end{proof}

Note that for the pairing $(s,i)$ defined in the lemma, we now know that both $s$ and $i$ are nonzero by parts (ii) and (i) respectively. 
\\\\
In order to contradict the existence of $F$ in $S\times I\setminus L$, we need an intersection curve that is trivial on $S_+$. By Lemma~\ref{lem:trivialIffTrivial}, it suffices to show that there exists a curve that is trivial on $F$.

\begin{lemma}
\label{lem:existsTrivialCurve}
There exists an intersection curve that is trivial on $F$.
\end{lemma}

\begin{proof}
First note that Lemma~\ref{lem:existsTrivialCurve} holds in case (1), as all curves on a sphere are trivial. Now consider cases (2) and (3), and suppose $F$ is a torus or an annulus.
\\\\
Consider $F$ with all intersection curves both in $F\cap S_+$ and $F\cap S_-$ projected onto it. All saddles correspond to quadrilaterals, which we collapse to vertices to obtain a 4-regular graph on $F$, as in Figure~\ref{fig:EulerCharacteristic}. Notably, for any two adjacent faces in this graph, one corresponds to a region which is contained strictly above $S_+$, and the other to a region contained strictly below $S_-$.\\\\

\begin{figure}[h]
\centering
\begin{subfigure}{0.45\textwidth}
\includegraphics[scale=0.1]{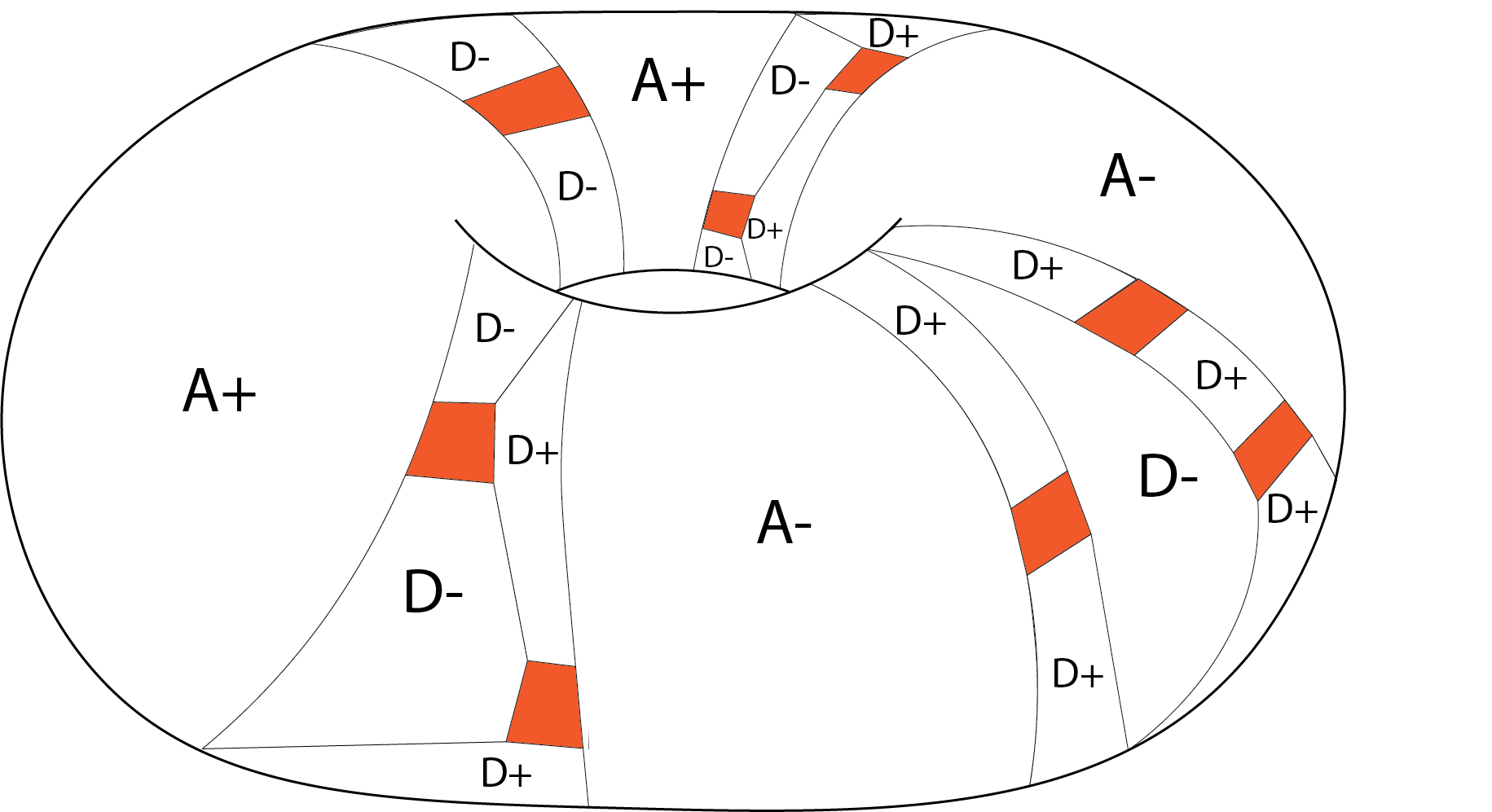}
\end{subfigure}
\begin{subfigure}{0.45\textwidth}
\includegraphics[scale=0.1]{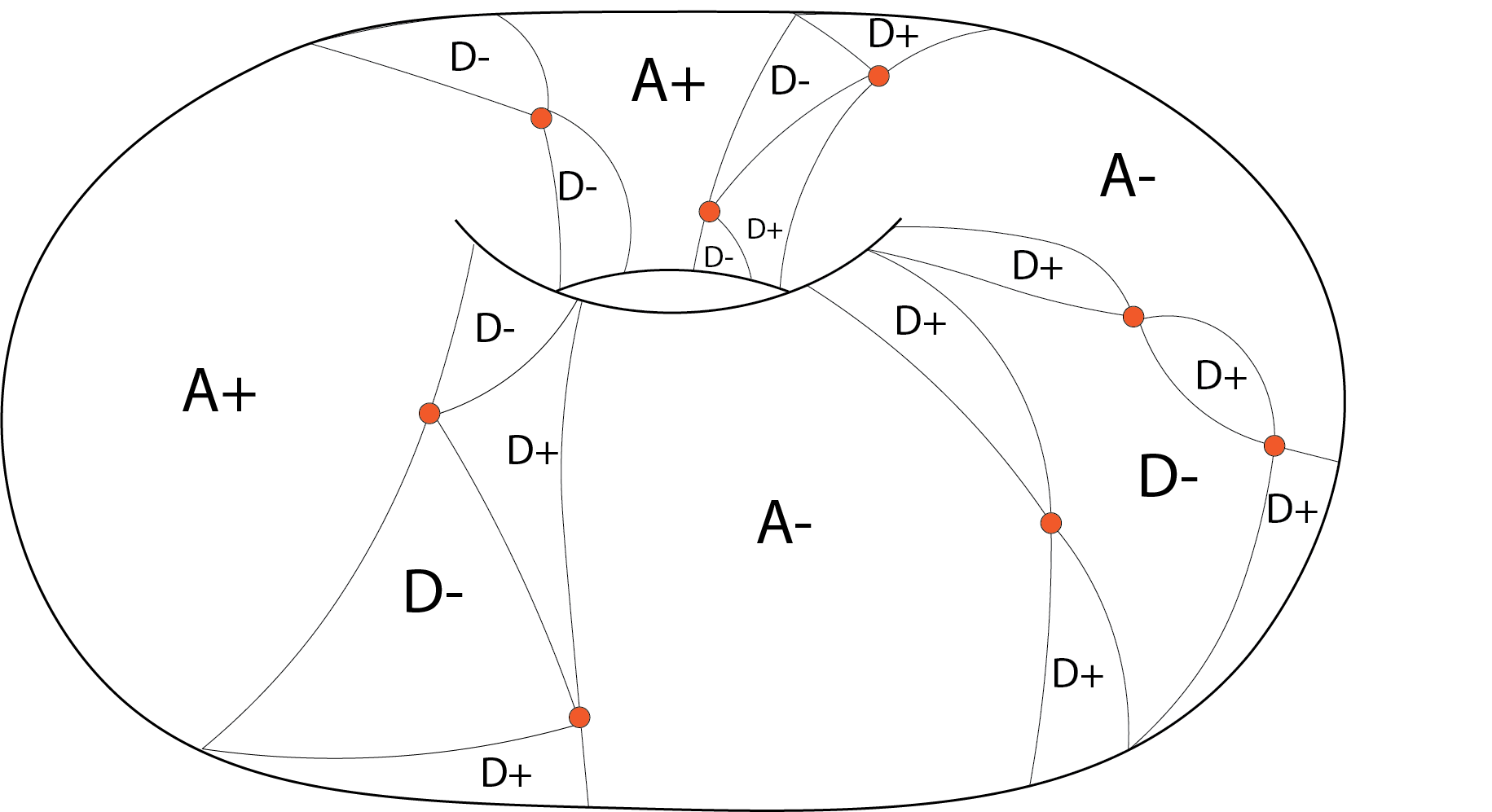}
\end{subfigure}
\caption{Torus $F$ with projection of intersection curves from $F\cap S_+$ and $F\cap S_-$. Annular regions denoted $A$, disk regions denoted $D$; regions contained above $S_+$ denoted $+$, regions contained below $S_-$ denoted $-$. Quadrilateral saddles (left) collapsed to vertices (right). }
\label{fig:EulerCharacteristic}
\end{figure}

We would like to show that at least one of the complementary regions of this graph is a disk, whose boundary corresponds to a trivial intersection curve of either $F \cap S_+$ or $F \cap S_-$. If none of the faces formed by the $4$-regular graph on $F$ are disks, then all faces must have nonpositive Euler characteristic contribution. 
We have:
$$\mathcal{V}=\text{number of saddles}>0$$
$$\mathcal{E}=\frac{4\mathcal{V}}{2}=2\mathcal{V}$$
$$\chi(F) = \mathcal{V} - \mathcal{E} +\mathcal{F} \leq \mathcal{V} - \mathcal{E} = \mathcal{V}-2\mathcal{V}<0$$
But the Euler characteristic of $F$ (a torus or an annulus) is $0$, contradiction. Thus there must be a disk region in the complement of the graph on $F$, which is bounded by a trivial intersection curve $\alpha$ in either $F\cap S_+$ or $F\cap S_-$. For convenience, we assume it is in $F\cap S_+$.
\end{proof}

Now we show that no trivial curve $\alpha$ intersects a bubble $B$ on both sides, such that the disk on $S_+$ bounded by $\alpha$ contains $L\cap \pa B_+$.

\begin{lemma}
\label{lem:noCurveBubbleBothSides}
Let $\alpha$ be an arc of an intersection curve in $F \cap S_+$ that begins and ends on different sides of a bubble $B$. Then $\alpha \cup \pa B_+$ must contain a nontrivial simple closed curve on $S_+$.
\end{lemma}

\begin{proof}
Of all arcs of $F \cap S_+$ that satisfy all the hypotheses, choose $\alpha$ to be the one that intersects $\pa B_+$ closest to the overstrand of $B$. Assume $\alpha \cup \pa B_+$ does not contain a simple closed curve that is nontrivial on $S_+$. Any other intersection curve $\alpha'$ between the two arcs formed by $\alpha$ inside $\pa B_+$ must pass through $\pa B_+$ at least twice on one side of $\pa B_+$. Clearly $\alpha' \cup \pa B_+$ also does not contain a simple closed curve that is nontrivial on $S_+$, so by Lemma~\ref{lem:isotopyBubbleCleanUp}~(iii), no such $\alpha'$ exists. It follows that $\alpha$ is the innermost intersection curve inside $\pa B_+$. Then $\alpha$ corresponds to one saddle $\sigma$ in $B$. 
\\\\
Now we again use the idea of the arc $\mu\subset F$ running along the intersection curve $\alpha$ as in the proof of Lemma~\ref{lem:isotopyBubbleCleanUp} and isotope it together with $F$ towards the bubble $B$. This would allow us to find a once-punctured disk with boundary in $\sigma\cup \mu$ on $F$. If there are no other intersection curves in the region bounded by $\sigma\cup \mu$, we can pull $\mu$ towards the bubble without obstruction so that it sits over $B$.
\\\\
The only case left to consider is when there exist other intersection curves in the region bounded by $\sigma\cup \mu$. For each such curve $\beta$, since $\alpha$ is trivial, $\beta$ is also trivial. Then all intersection curves in the region bounded by $\sigma\cup \mu$ are trivial and bound disks on $F$ to at least one side. Therefore, we can always isotope $\mu$ with $F$ along these disks so that $\mu$ sits over $B$. As $\alpha$ is the innermost intersection curve in $B$, there exists a circle in $\mu\cup\sigma$ wrapping once around the overstrand of $B$, thereby bounding a meridional compression disk for $F$, a contradiction. Therefore, no intersection curve satisfying the hypotheses can intersect a bubble on both sides.
\end{proof}


\begin{lemma}
\label{lem:noMerIncSurface}
Let $L$ be a fully alternating link in $S \times I$ and let $F$ in $S\times I\setminus L$ be a sphere, a torus, or an annulus with boundaries strictly on $S\times I$. Then $F$ cannot be essential and meridionally incompressible.
\end{lemma}

\begin{proof}
By Lemma~\ref{lem:existsTrivialCurve}, we know that there exists a trivial intersection curve. Consider a trivial intersection curve $\alpha$ that is innermost on $S_+$, bounding disk $D$ on $S_+$. By Lemma~\ref{lem:isotopyBubbleCleanUp}, we know that $\alpha$ intersects at least one bubble.
\\\\
As the projection is fully alternating, any time $\alpha$ enters a region through a bubble such that $\alpha$ is to the right (similarly left) of the overstrand, $\alpha$ must leave the region through a bubble such that it is to the left (similarly right) of the overstrand, as in Figure~\ref{fig:punchline}.
\\

\begin{figure}[h]
\centering
\includegraphics[scale=0.3]{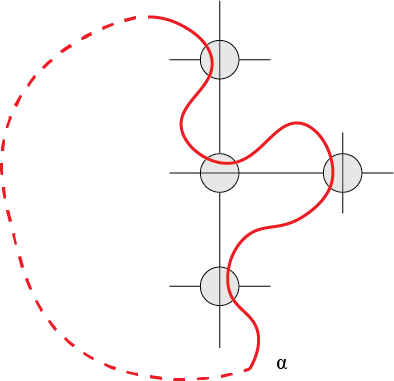}
\caption{Innermost trivial curve $\alpha$.}
\label{fig:punchline}
\end{figure}

Suppose that the disk bounded by $\alpha$ on $S_+$ contains the overstrand of a bubble that it intersects. Then because there are no other curves of intersection in $D$, $\alpha$  must hit the other side of that bubble, contradicting  Lemma~\ref{lem:noCurveBubbleBothSides}. 
\\\\
Hence the disk bounded by $\alpha$ on $S_+$ does not contain the overstrand of any bubble that it intersects.  It follows that for some bubble $B$, the curve $\alpha$ passes through one side of $B$ such that the overstrand is on the left (similarly right), and then passes through the same side of $B$ such that the overstrand is on the right (similarly left), with the disk side of $\alpha$ being between these two passes. By Lemma~\ref{lem:isotopyBubbleCleanUp}~(iii), $\pa B_+ \cup \alpha$ contains a simple closed curve that is nontrivial on $S_+$, contradicting that $\alpha$ is trivial on $S_+$. 
\\

\begin{figure}[h]
\centering
\includegraphics[scale=0.3]{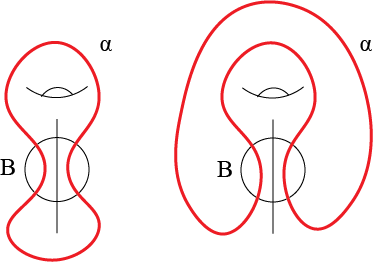}
\caption{Two cases when $\pa B_+ \cup \alpha$ contains a simple closed curve that is nontrivial on $S_+$}
\label{fig:curveBothSides}
\end{figure}

\end{proof}

Finally, in the case that $L$ is prime, we consider essential spheres and tori in $S \times I \setminus L$ that are meridionally compressible. This will eliminate all essential spheres and tori.

\begin{lemma}
\label{lem:noMerComSphereTorus}
Let $L$ be a prime, fully alternating link in $S \times I$ and let $F$ be a sphere, a torus or an annulus with boundaries strictly on $S \times I$. Then $F$ cannot be essential and meridionally compressible.
\end{lemma}

\begin{proof}
It is clear that there are no meridionally compressible spheres. Let $F$ be an incompressible, meridionally compressible torus that is not boundary--parallel. Then a meridional compression yields a twice-punctured sphere $F'$ in $S \times I \setminus L$ that bounds a ball containing a nontrivial portion of $L$. This contradicts the assumption that $L$ is prime in $S \times I$.
\\\\
If $F$ is an essential meridionally compressible annulus with boundary in $S \times I$, apply the meridional compression. This results in two once-punctured disks, each with boundary on $S\times I$. But by incompressibility of $S$ in $S \times I$, the boundary of the disk must be trivial on $S$. Hence, we can construct a sphere punctured once by $L$, a contradiction. 
\end{proof}

\section{Eliminating essential annuli}
\label{sec:annulus}
In this section we eliminate essential annuli, which will allow us to prove Theorem~\ref{thm:main}. Let $A$ be an essential annulus in $M=S\times I\setminus \mathring{N}(L)$. As our annulus $A$ has two boundary components, exactly one of the following holds:
\begin{enumerate}
\item $A$ has boundary strictly on $\pa (S\times I)$.
\item $A$ has boundary strictly on $\pa N(L)$.
\item $A$ has boundary on both $\pa (S\times I)$ and $\pa N(L)$.
\end{enumerate}

\medskip

By Lemmas \ref{lem:noMerIncSurface} and \ref{lem:noMerComSphereTorus}, Case (1) has already been eliminated. 

\begin{lemma}
\label{lem:noAnnulusOnL}
There are no essential annuli in $M$ with both boundary components on $\pa N(L)$.
\end{lemma}

\begin{proof}

Assume that $A$ is an essential annulus with boundary strictly on $\pa N(L)$.
Define $Q$ to be the neighborhood of the union of $A$ with the torus or tori of $L$ on which the annulus has boundary components. As in the proof of Lemma 1.16 in~\cite{hatcher}, one of the following must hold:
\begin{enumerate}
\itemsep0em
\item[(a)] $A$ has boundary on two different tori, and $Q$ is the product of a twice-punctured disk with $S^1$.
\item[(b)] $A$ has both boundary components on one torus, and $Q$ is the product of a twice-punctured disk with $S^1$.
\item[(c)] $A$ has both boundary components on one torus, and $Q$ is a circle bundle over a punctured Möbius band.
\end{enumerate}
Observe that cases (a) and (b) have three torus boundaries and case (c) has two torus boundaries. We have previously shown in Lemmas~\ref{lem:noMerIncSurface} and~\ref{lem:noMerComSphereTorus} that our manifold contains no essential tori; thus the tori of $\partial Q\setminus \partial M$ must be either boundary-parallel or compressible. In all cases, we show that constructing $M$ from $Q$ considerably restricts the structure that $M$ can take, from which we derive a contradiction.
\\\\
There are up to two tori in the boundary components of $\pa Q\setminus \pa M$. If a torus $T$ of $\partial Q\setminus \partial M$ is boundary-parallel, it follows that $T$ has an external collar $T\times I$, which can be glued to $Q$ via $T$. We note this does not change $Q$ topologically.
\\\\
If one of the tori $T$ of $\pa Q\setminus \pa M$ is compressible, consider a compression disk $D$ for $T$. A compression of $T$ along $D$ yields a sphere $R$, and as we have shown that $M$ does not contain essential spheres by Lemma~\ref{lem:noMerIncSurface}, the sphere $R$ must bound a ball $B$. Therefore, if we resurger $R$ to get $T$ back, $B$ becomes a solid torus. 
\\\\
Hence $M$ comes from $Q$ by gluing either a $T\times I$ or a solid torus to each component of $\pa Q\setminus \pa M$. However, gluing a solid torus to any boundary component of $Q$ lowers its number of boundary components to less than three. Note $M$ must have at least three boundary components: two from $S\times I$, plus one for each component of $L$. Therefore, only copies of $T\times I$ are glued onto $Q$, so $M$ is homeomorphic to $Q$.
\\\\
Now, because $M$ as constructed from $Q$ has only torus boundary components, the only possibility is for $S$ to have genus $g=1$. Note that in case (c), the entire manifold only has two boundary components, thereby eliminating case (c) from consideration.
\\\\
We observe that cases (a) and (b) are equivalent up to homeomorphism. Thus $M$ is a twice-punctured disk crossed with $S^1$: a manifold with three boundary components. These three boundary components correspond to the boundaries of $S\times I\setminus \mathring{N}(L)$, implying the link $L$ has only one component, corresponding to one of the punctures crossed with $S^1$. Observe that $L$ admits a projection without crossings onto either of the boundaries of $S\times I$. Thus the crossing number of $L$ in $S\times I$ is $0$. However, the crossing number $c(L)$ of an alternating knot on a given surface $S$ in $S\times I$ is given by its reduced alternating projection (see~\cite{crossing number}), which is assumed to have at least one crossing, a contradiction.
\end{proof}

\begin{lemma}
\label{lem:noAnnulusOnLS}
There are no essential annuli in $M$ with one boundary component on $\pa N(L)$ and one boundary component on $\partial (S\times I)$.
\end{lemma}
\begin{proof}

Assume that $A$ is an essential annulus with one boundary component on a torus of $\partial N(L)$ and one boundary component on $\partial(S\times I)$. A similar proof to that of Lemma~\ref{lem:noAnnulusOnL} shows that $S$ has genus at least 2.
\\\\
Reflect $M$ through its non-torus boundaries to get a second copy $M^R$ of $M$, and let $M'$ be the double of $M$ given by $M \cup_{\pa M}M^R$ with corresponding non-torus boundary components identified. We note that $M'= (S\times S^1)\setminus (\mathring{N}(L)\cup\mathring{N}(L^R))$ and that $\pa M'$ consists of the boundaries of the neighborhoods of $L$ and $L^R$. Note $A$ and its corresponding annulus $A^R$ are now glued along one of their boundaries to form $A' = A \cup A^R$, an essential annulus in $M'$ with its two boundary components on $\pa N(L)$ and $\pa N(L^R)$, respectively.
\\\\
We first note that we can extend the fact that there are no essential spheres or tori in $S \times I \setminus L$ to $M'$.
For any such sphere or torus could not be entirely contained in either $M$ or $M^R$ by Lemmas \ref{lem:noMerIncSurface} and \ref{lem:noMerComSphereTorus}. By the incompressibility of $\pa S \times I$, this eliminates spheres and any such torus must intersect $M$ and $M^R$ in essential annuli with their boundaries on $\pa S \times I$. But these were also eliminated in Lemmas \ref{lem:noMerIncSurface} and \ref{lem:noMerComSphereTorus}.
\\\\
Now, we prove the lemma when $L$ consists of $k$ components, $k\geq 2$. We observe that our new manifold $M'$ has $2k\geq 4$ boundary components. We have already shown that a manifold with no essential tori, which contains an essential annulus meeting only torus components of the boundary, has at most three boundary components, a contradiction.
\\\\
Now let $L$ have one component. Then the annulus $A'$ has boundary on $\pa N(L)$ and $\pa N(L^R)$, which corresponds to case (a) of the proof of Lemma~\ref{lem:noAnnulusOnL}. Let $Q$ be the neighborhood of the annulus $A'$ and the tori $\pa N(L)$ and $\pa N(L^R)$. Then the outer torus of $Q$ cannot be essential, so it must be either compressible or boundary-parallel. Note that if it is boundary-parallel, the entire manifold $M'$ must have 3 boundary components. But $M'$ was obtained by doubling $M$, and hence must have an even number of boundary components, a contradiction.
\\\\
If the outer torus of $Q$ is compressible, then $M$ is obtained by gluing a solid torus to the outside of $Q$ as in the proof of Lemma~\ref{lem:noAnnulusOnL}. We know that $Q$ is a twice-punctured disk crossed with $S^1$, which we can think of as $T \times I$ minus a nontrivial simple closed curve lying in $T\times\{\frac{1}{2}\}$. Gluing a solid torus to $T \times I$ along one of its boundaries gives us a solid torus, so the result is equivalent to a solid torus with a neighborhood of a $(p,q)$-curve removed, denoted $V_{p,q}$.
\\

We wish to distinguish between $V_{p,q}$ and $M'$ to derive a contradiction. 
For $V_{p,q}$, we first find the fundamental group. Cutting $V_{p,q}$ open along an annulus with both boundaries on the neighborhood of the $(p,q)$-curve, we get a solid torus with generator $\gamma$ and a $T\times I$ with generators $\alpha$ and $\beta$. We notice that $\gamma$ wraps $q$ times around the solid torus inside it, as in the figure below.
\\

\begin{figure}[h]
\centering
\includegraphics[width=4.5cm, height = 4.125cm]{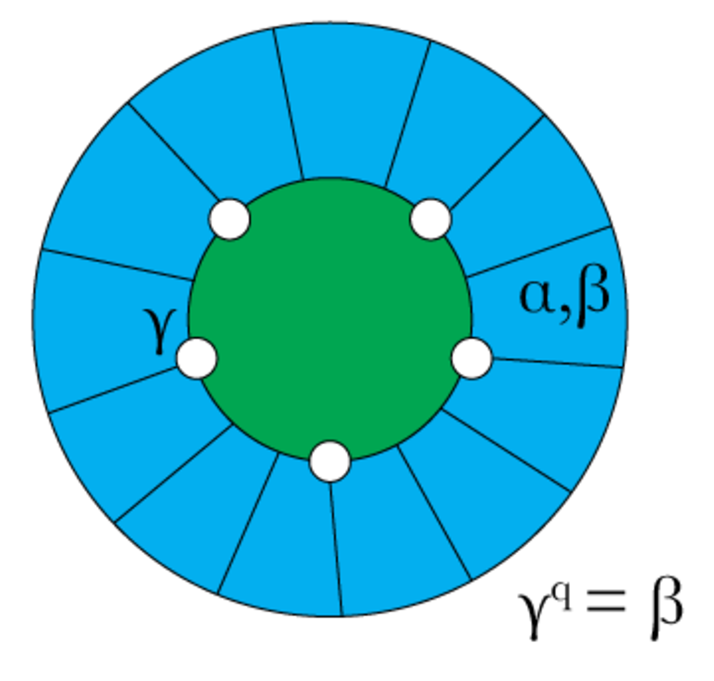}
\caption{The cross-section of a solid torus minus the $(p,q)$-curve. Cut along the annulus along the $(p,q)$-curve and we get a solid torus glued to $T\times I$ along the $(p,q)$-annulus.}
\label{fig:pqthing}
\end{figure}

Then, $\pi_1(V_{p,q})=\langle \alpha,\beta,\gamma|\gamma^q=\alpha,\alpha\beta=\beta\alpha\rangle$ and $H_1(V_{p,q})=\mathbb{Z}^2$. On the other hand, we certainly know that $M'$ contains a copy of $S$ as an incompressible genus $g$ surface because $S\times S^1$ does. This implies that $H_1(M')$ contains a $\mathbb{Z}^{2g}$ subgroup, and since $g\geq 2$, we have a contradiction.
\end{proof}

Thus our manifold can contain no essential annuli. Finally, we prove Theorem~\ref{thm:main}, i.e., a prime, fully alternating link in $S\times I$ is hyperbolic. 

\begin{proof}[Proof of Theorem~\ref{thm:main}] We check the conditions in Thurston's Hyperbolization Theorem~\cite{thurston}. By Lemmas~\ref{lem:noMerIncSurface} and~\ref{lem:noMerComSphereTorus}, there are no essential spheres or tori in $S\times I\setminus \mathring{N}(L)$. By Lemmas~\ref{lem:noMerIncSurface},~\ref{lem:noMerComSphereTorus},~\ref{lem:noAnnulusOnL} and~\ref{lem:noAnnulusOnLS}, there are no essential annuli in $S\times I\setminus \mathring{N}(L)$. Therefore, a prime fully alternating link $L$ in $S\times I$ is hyperbolic.
\end{proof}

\section{Primeness of alternating links in $S \times I$}
\label{sec:prime}

In this section we prove Theorem~\ref{thm:primeiffobv}, i.e. a link in $S\times I$ with a reduced fully alternating projection on $S$ is prime if and only if the projection is obviously prime. \\

Suppose a reduced, fully alternating projection $P$ of $L$ on $S$ is not obviously prime. Then we show that $L$ is not prime in $S\times I$.  By definition, we can find a twice-punctured circle in $P$, which bounds a disk containing at least one crossing of $P$. A neighborhood of the disk is a ball $W$ in $S\times I$ contining an alternating portion of $L$. Since the least number of crossings of an alternating link occurs in any reduced alternating projection (\cite{kauffman},\cite{murasugi},\cite{thistlethwaite}), $W$ contains a nontrivial portion of $L$.
Since the boundaries of $S \times I$ are outside $W$, we know that $\pa W$ must be essential, which shows that $L$ is not prime.
\\\\
Suppose now that $L$ is a non-prime, fully alternating link in $S\times I$ with a reduced fully alternating projection $P$ on $S$. We want to show that $L$ is not obviously prime in $P$. Let $F$ be an essential twice-punctured sphere in $S \times I \setminus L$. We can assume that $F$ is meridionally incompressible, since otherwise a meridional compression would generate two twice-punctured spheres, and we can iterate this process until we obtain an essential, meridionally incompressible twice-punctured sphere.
\\\\
As in Section~\ref{sec:SphereTorus}, we project $L$ onto $S$, place bubbles at each crossing, and define the surfaces $S_+$, $S_-$ and $S_0$. We consider the intersection curves $F \cap S_+$ as before, although now we allow two places where intersection curves cross the link, created by the punctures.

\begin{lemma}
\label{lem:primeLemmas}
Let $F$ be an essential, meridionally incompressible twice-punctured sphere in $S \times I \setminus L$. Then there exists an isotopy of $F$ such that the following are true:
\begin{enumerate}[label=(\roman*)]
\item Every intersection curve of $F \cap S_+$ is trivial on both $F$ and $S_+$.
\item Every intersection curve in $F \cap S_+$ intersects a bubble or the link $L$ at least once.
\item There are no easily removable intersections of $F$ with bubbles as in Figure~\ref{fig:easilyRemovable}.
\end{enumerate}
Further, let $\alpha$ be an arc of an intersection curve in $F \cap S_+$, and let $B$ be a bubble that $\alpha$ passes through, such that $\pa B_+ \cup \alpha$ does not contain a simple closed curve that is nontrivial on $S_+$.
\begin{enumerate}[resume,label=(\roman*)]
\item For any such pair $B$ and $\alpha$, the arc $\alpha$ does not pass through the same side of $\pa B$ more than once.
\item For any such pair $B$ and $\alpha$, the arc $\alpha$ does not pass through both sides of $\pa B$.
\end{enumerate}
\end{lemma}

\begin{figure}[h]
\includegraphics[scale=0.3]{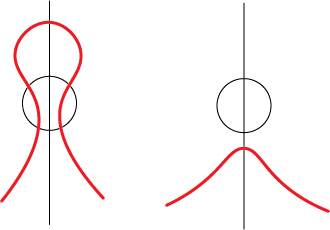}
\centering
\caption{Easily removable intersection of $F$ and a bubble.}
\label{fig:easilyRemovable}
\end{figure}

\begin{proof}
\text{} 
\begin{enumerate}[label=(\roman*), wide]
\item By the argument from the proof of Lemma~\ref{lem:trivialIffTrivial}, observe that an intersection curve is trivial on $F$ if and only if it is also trivial on $S_+$, where we define a curve to be nontrivial on $F$ if it separates the two punctures. Note that the argument for cases (2) and (3) applies because there are no essential spheres in $S \times I \setminus L$ by Lemma~\ref{lem:noMerIncSurface}.
\\\\
Moreover, we claim that any intersection curve on $F\cap S_+$ is trivial on $S_+$. Suppose some intersection curve $\alpha$ on $F\cap S_+$ is nontrivial on $S_+$. Since $F$ is a twice-punctured sphere, $\alpha$ bounds either a disk or a once-punctured disk on $F$ to one side. Fill $L$ back into $S\times I\setminus L$ and consider the closure $D$ of the disk or the once-punctured disk bounded by $\alpha$. Then $D$ is a compression disk of $S_+$ in $S\times I$, contradicting that $S_+$ is incompressible in $S\times I$. Thus all intersection curves are trivial on $S_+$. Note that this argument also applies to intersection curves $F \cap S'$, where $S'$ is defined to be $S_0$ where the equatorial disk in each bubble replaced by any combination of upper and lower hemispheres.
\\
\item As in the proof of Lemma~\ref{lem:isotopyBubbleCleanUp}, we associate an ordered pair $(s,i)$ to each embedding of $F$ prior to isotopy, where $s$ is the number of saddles in $F$ and $i$ is the number of intersection curves in $F \cap S_+$. Suppose we pick $F$ to be the embedding such that its ordered pair $(s,i)$ is minimal under lexicographical ordering. Then it follows from the same argument that we can remove intersection curves that do not hit any bubbles and are not punctured by a link component.
\\
\item We can assume that there are no easily removable intersections of $F$ with bubbles as in Figure~\ref{fig:easilyRemovable}, for otherwise we could isotope $F$ to eliminate a saddle.
\\
\item Again it follows from the same argument that for any pair $B$ and $\alpha$, if $\pa B_+ \cup \alpha$ does not contain a simple closed curve that is nontrivial on $S_+$, then $\alpha$ does not intersect the same side of $\pa B$ more than once.
\\
\item As in the proof of Lemma~\ref{lem:noCurveBubbleBothSides},  for any such pair $B$ and $\alpha$, if $\pa B_+ \cup \alpha$ does not contain a simple closed curve that is nontrivial on $S_+$, we get a once-punctured disk $D$ with $D \cap F = \pa D$. Since we have eliminated easily removable intersections of $F$ with the bubbles as in Figure~\ref{fig:easilyRemovable}, any such once-punctured disk $D$ must be a meridional compression, a contradiction. Therefore $\alpha$ does not intersect both sides of $\pa B$.
\end{enumerate}
\end{proof}

Now we consider each intersection curve $\alpha$ of $F \cap S_+$. Note that $\alpha$ may hit bubbles and it can hit $L$ at most twice (which correspond to punctures of $F$).

\begin{lemma}
\label{lem:atLeastTwice}
Every intersection curve of $F \cap S_+$ must intersect $L$ at least twice.
\end{lemma}

\begin{proof}
First observe that in a fully alternating projection, if $\alpha$ intersects $L$ it enters an adjacent region. If $\alpha$ entered the initial region through a bubble with the overstrand on the right (respectively left), it follows that when $\alpha$ leaves this adjacent region through a bubble, the overstrand is again on the right (respectively left). Hence the sum of the number of intersections of $\alpha$ with $L$ and of $\alpha$ with bubbles must be even.
\\\\
Suppose $\alpha$ does not hit any bubbles. Then $\alpha$ must intersect $L$, since we eliminated intersection curves that do not hit bubbles or $L$. So, $\alpha$ must intersect $L$ an even number of times, which implies $\alpha$ must intersect $L$ at least twice.
\\\\
Now suppose $\alpha$ only intersects a single bubble $B$ and does so once. Then on $S_-$, there is an arc $\beta$ of an intersection curve in $F \cap S_+$ that hits both sides of $\pa B$. By Lemma~\ref{lem:primeLemmas}~(v), $\pa B_+ \cup \beta$ must contain a nontrivial curve on $S_-$ (see Figure~\ref{fig:bubbleonce}(B)). Then $\alpha$ must have been a nontrivial curve on $S_+$, which contradicts that all intersection curves are trivial on $S_+$.
\\

\begin{figure}[h]
\centering
\begin{subfigure}{0.45\textwidth}
\centering
\includegraphics[scale=0.3]{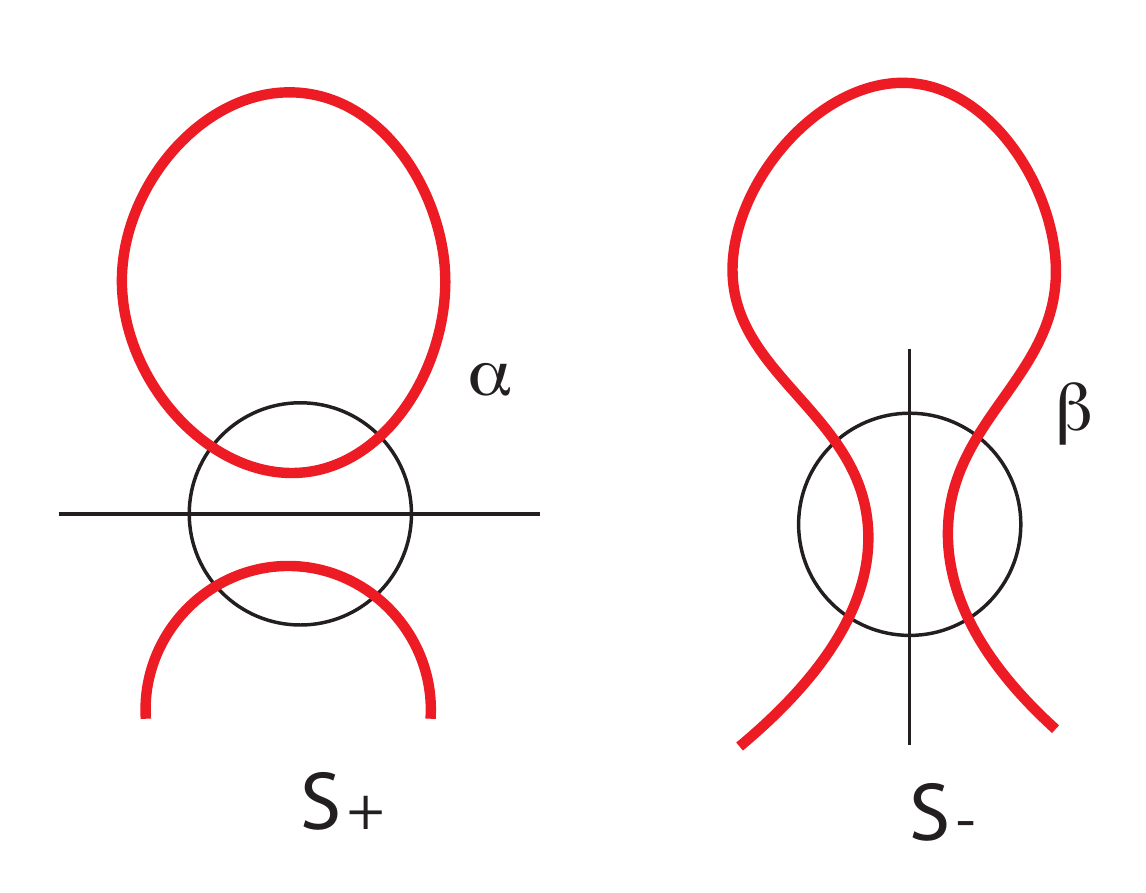}
\caption{Intersection curve in $F \cap S_-$ which passes through both sides of a bubble $B$.}
\label{fig:primeDiagram1}
\end{subfigure}
\begin{subfigure}{0.45\textwidth}
\centering
\includegraphics[scale=0.3]{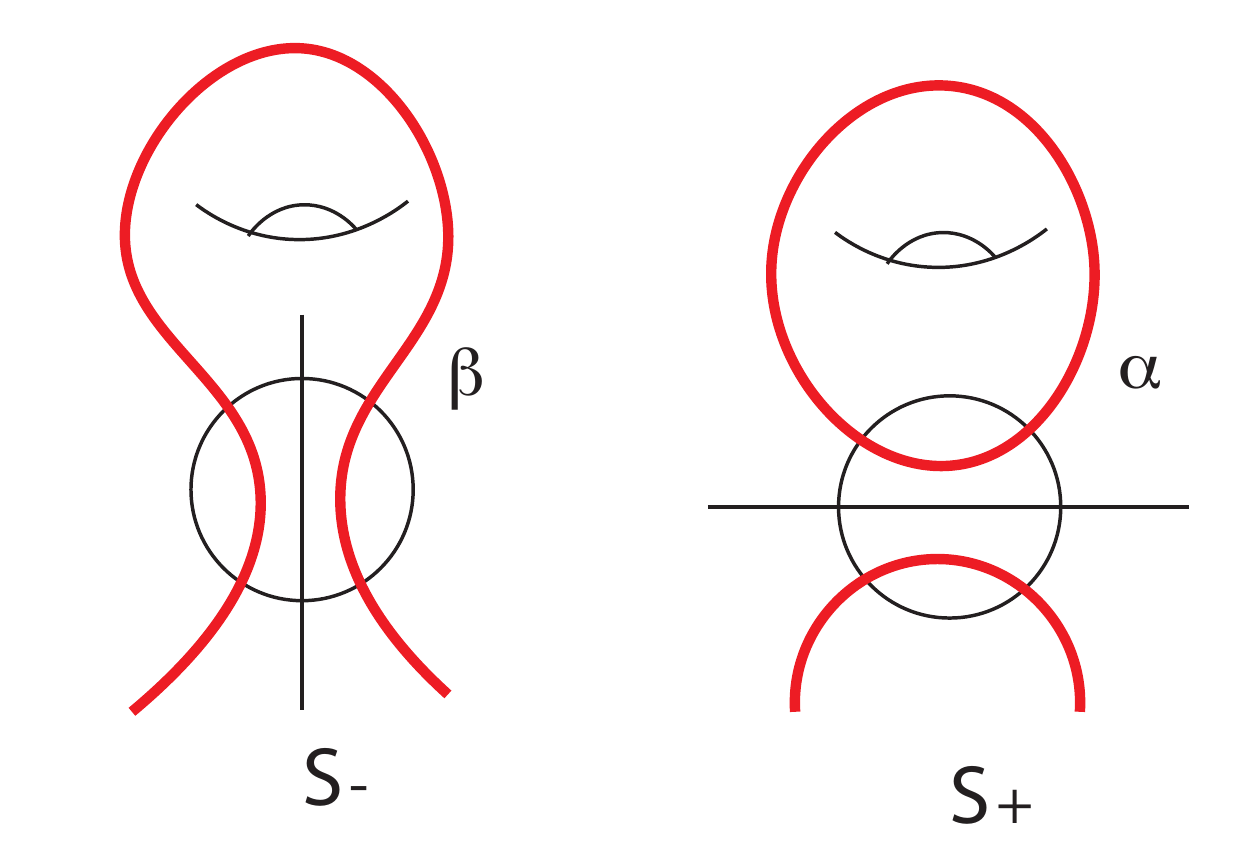}
\caption{If $\pa B_+ \cup \beta$ contains a nontrivial curve, $\alpha$ must have been nontrivial.}
\label{fig:primeDiagram1b}
\end{subfigure}
\caption{When $\alpha$ hits only one bubble once.}
\label{fig:bubbleonce}
\end{figure}

Finally, suppose $\alpha$ intersects bubbles more than once, such that between any two bubbles intersected by $\alpha$, there exists an arc of $\alpha$ that connects the two bubbles without any punctures in between. Let $D$ be the disk bounded by $\alpha$ on $S_+$. Since $L$ is alternating, $D$ contains the overstrand of at least one bubble.
\\\\
Consider each intersection of $\alpha$ with a bubble $B$, such that $\alpha$ bounds the overstrand of $B$ to the interior of $D$. For each such intersection, there are two intersection curves given by $F\cap S_-$ entering $D$ (Figure~\ref{fig:primeDiagram2}). Thus there are an even number of intersection curves entering $D$ on $S_-$, all of which connect within $D$, and thus form arcs with endpoints on $\alpha$. Of these arcs, choose the arc $\beta$ that is outermost in $D$. There are three possibilities: $\beta$ connects two sides of one bubble, $\beta$ connects two different bubbles with a bubble in between, or $\beta$ connects two curves passing through one side of one bubble.
\\

\begin{figure}[h]
\includegraphics[scale=0.5]{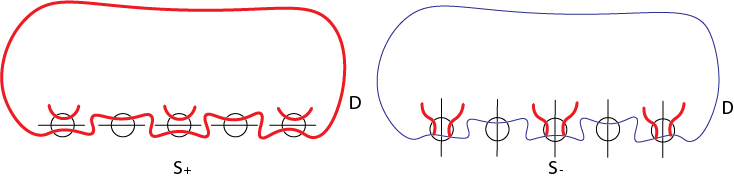}
\centering
\caption{Intersection curves of $F \cap S_-$ entering $D$.}
\label{fig:primeDiagram2}
\end{figure}

Assume $\beta$ intersects both sides of one bubble. We see from Lemma~\ref{lem:primeLemmas}~(v), that this can only occur when $\pa B_+\cup\beta$ contains a simple nontrivial curve; but this contradicts that $\beta$ is contained in $D$.
\\\\
Assume $\beta$ connects two different bubbles with a bubble $B$ in between (Figure~\ref{fig:primeDiagram3}). Again, it follows that $\beta$ must hit both sides of $B$, and as $\beta$ is contained in $D$, this contradicts Lemma~\ref{lem:primeLemmas}~(v).
\\

\begin{figure}[h]
\includegraphics[scale=0.5]{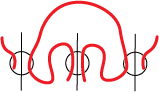}
\centering
\caption{When $\beta$ connects two different bubbles with a bubble $B$ in between.}
\label{fig:primeDiagram3}
\end{figure}

Finally, assume $\beta$ connects two intersection arcs both passing through the same side of the same bubble (Figure~\ref{fig:primeDiagram4}). As $\pa B_+ \cup \beta$ does not contain a simple closed curve that is nontrivial on $S_-$, this contradicts that all such curves were removed with Lemma~\ref{lem:primeLemmas}~(iv).
\begin{figure}[h]
\includegraphics[scale=0.5]{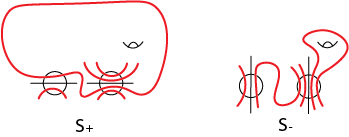}
\centering
\caption{When $\beta$ connects two strands passing through the same side of one bubble.}
\label{fig:primeDiagram4}
\end{figure}
\\\\
Therefore there are at least two intersections with $L$, which prevent the intersections with bubbles from being consecutive.
 \end{proof}

Now we are ready to show that $P$ is not obviously prime. This completes the proof of Theorem~\ref{thm:primeiffobv}, i.e. a fully alternating link $L$ in $S \times I$ with a reduced fully alternating projection $P$ is prime if and only if $P$ is obviously prime.

\begin{proof}[Completion of Proof of Theorem~\ref{thm:primeiffobv}]
Observe that since $F$ is a twice-punctured sphere, there can only be two intersections with $L$ in all intersection curves of $F \cap S_+$. By Lemma~\ref{lem:atLeastTwice}, $F\cap S_+$ consists of exactly one intersection curve $\alpha$, which intersects $L$ exactly twice. Since all intersection curves are trivial on $S_+$, $\alpha$ must be trivial on $S_+$.
\\\\
Suppose $\alpha$ hits a bubble $B$. Let $\sigma$ be the uppermost saddle of $B$. Then, since $\alpha$ is the only intersection curve, both arcs of $\sigma \cap \pa B$ must belong to $\alpha$. If an arc of $\alpha \setminus B$ satisfies the hypothesis of Lemma~\ref{lem:primeLemmas}~(v), then we obtain a contradiction.
\\\\
Suppose that no arc of $\alpha \setminus B$ satisfies the hypothesis of Lemma~\ref{lem:primeLemmas}~(v); in other words, suppose that for every arc $\beta$ of $\alpha \setminus B$, $\pa B_+ \cup \beta$ contains a simple closed curve that is nontrivial on $S_+$.
Let $S'$ be the surface obtained from $S_+$ by replacing the upper hemisphere of $B$ with the lower hemisphere. 
Then $\alpha\cup \pa \sigma$ contains a simple closed curve $\gamma$ that is nontrivial on $S'$. But $\gamma$ is an intersection curve  of $F$ and the surface $S'$. Therefore it is trivial on $F$, which contradicts the fact that $S'$ is incompressible in $S\times I$.(See Figure \ref{fig:primeBubbleBothSides}).
\\
\begin{figure}[h]
\includegraphics[scale=0.3]{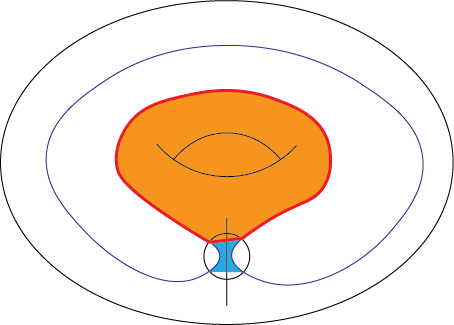}
\centering
\caption{When $\alpha \cup \pa \sigma$ contains a nontrivial simple closed curve.}
\label{fig:primeBubbleBothSides}
\end{figure}

Thus $\alpha$ does not intersect any bubbles, so $\alpha$ is a twice-punctured circle which bounds a disk $D$ on $S_+$. Since $F$ is an essential twice-punctured sphere, $D$ must contain at least one crossing of $L$. Thus $P$ is not an obviously prime projection of $L$.
\end{proof}

\section{Alternating links in more general 3-manifolds}
\label{sec:extensions}
Given an orientable hyperbolic 3-manifold $M$, we know that $M$ does not contain any essential spheres, tori  or annuli by Thurston's theorem. Furthermore, there are no projective planes or Klein bottles, because in an orientable 3-manifold, the boundary of the neighborhood of a projective plane is an essential sphere, and the boundary of the neighborhood of a Klein bottle is an essential torus.

\begin{lemma}\label{lem:nbd}
Let $S$ be a surface with negative Euler characteristic and let $N$ be an $I$-bundle over $S$. Let $L$ be a prime link in $N$ with a fully alternating projection on $S$. Then, $N\setminus L$ is hyperbolic.
\end{lemma}

\begin{proof}
We split into two cases: $N$ is orientable or $N$ is non-orientable. First suppose $N$ is orientable. By Theorem \ref{thm:main}, we already know that the lemma holds when $S$ is orientable, in which case $N=S\times I$. Hence, we only need to consider the case where $S$ is non-orientable. In this case, we use $N=S\widetilde{\times}I$ to denote the orientable twisted $I$-bundle over $S$. We know that we can always find an orientable double cover $\widetilde{N}=\widetilde{S}\times I$ of $N$, where $\widetilde{S}$ is orientable. We aim to show that the lift $\widetilde{L}$ of $L$ in $\widetilde{N}$ is fully alternating on $\widetilde{S}$ and prime in $\widetilde{N}$.
\\\\
To show that $\widetilde{L}$ is alternating, we can think of $S$ as a polygon $R$ with edges identified so that $\widetilde{S}$ is two copies of $R$ glued together with edges identified, denoted by $\widetilde{R}$. Certainly $\widetilde{L}$ is alternating in the interior of $\widetilde{R}$, since $L$ is alternating in the interior of $R$ and along the glued edges. And, because $L$ must alternate across identified edges of $R$, $\widetilde{L}$ must alternate across identified edges of $\widetilde{R}$. Thus, $\widetilde{L}$ is alternating on $\widetilde{S}$. We also know that disks in $S$ always lift to disks in $\widetilde{S}$, so $\widetilde{L}$ is fully alternating on $\widetilde{S}$.
\\\\
To show that $\widetilde{L}$ is prime in $\widetilde{N}$, consider an essential twice-punctured sphere $\widetilde{A}$ in $\widetilde{N}\setminus\widetilde{L}$. We know $\widetilde{A}$ must map to an immersed twice-punctured sphere $A$ in $N$ 
where any self-intersections are comprised of double curves. We want to construct from $A$ an embedded twice-punctured sphere in $N$ that is essential, contradicting the assumption that $L$ is prime in $N$.
\\\\
Any self-intersection curve on $A$ can be associated to two curves on $\widetilde{A}$ via the immersion. If both the curves associated to a self-intersection are trivial, then that self-intersection can be isotoped away since $\widetilde{N}$ does not contain any essential spheres.
\\\\
If both the curves associated to a self-intersection are non-trivial, we can perform surgery on $A$ by smoothing the intersecting sheets in two possible ways to eliminate the non-trivial self-intersections as shown in cross-section in Figure \ref{fig:smoothing}. Discarding the torus component in the second case, this results in two twice-punctured spheres $A_1$ and $A_2$ in $N$ on either side of the sheet containing the corresponding double-curve, with punctures coming from the link. Note that after smoothing $k$ double-curves associated to two non-trivial curves on $\widetilde{A}$ we obtain $2^k$ twice-punctured spheres, half of which lie on a single side of any smoothed double-curve and half of which lie on the other side. So, if $A$ has $k$ double-curves associated to two non-trivial curves on $\widetilde{A}$, after smoothing them all we have $2^k$ twice-punctured spheres $\{A_i\}$, each with no double curves that are associated to two trivial or two non-trivial curves on $\widetilde{A}$.
\\\\
\begin{figure}[h]
\includegraphics[scale=.5]{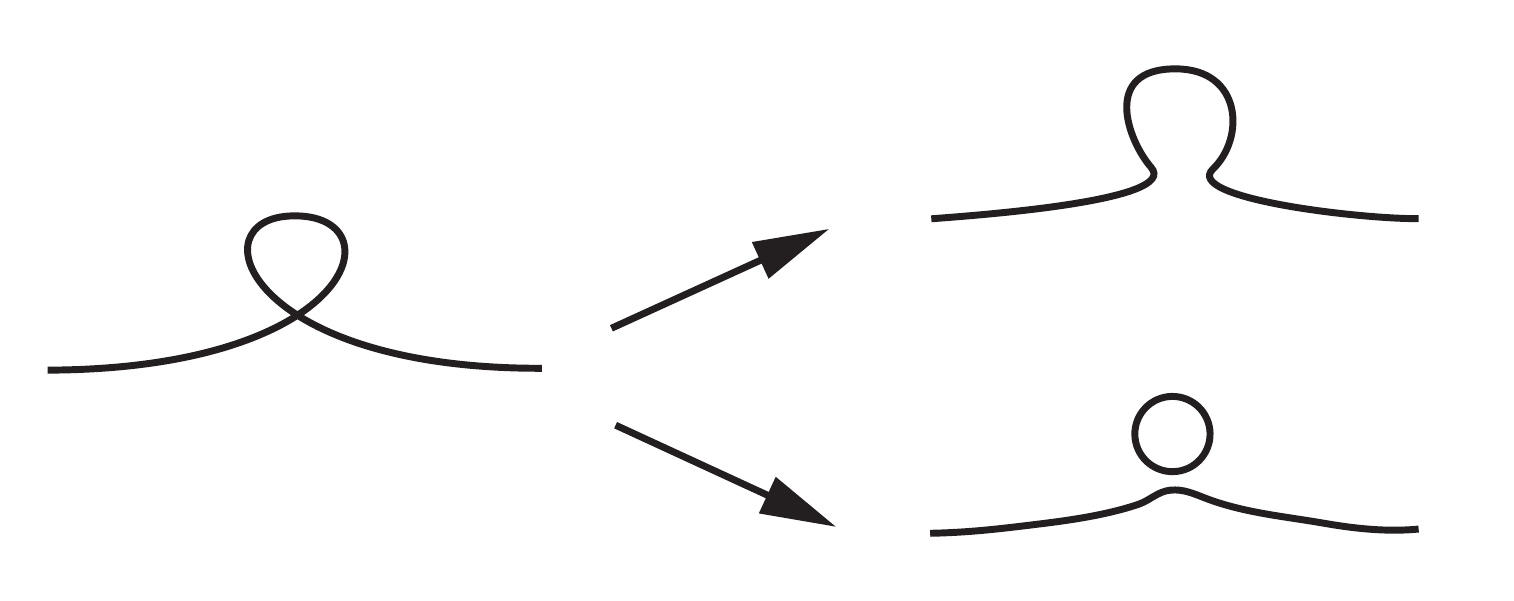}
\caption{Smoothing of a self-intersecting curve.}
\label{fig:smoothing}
\end{figure}
Finally, if a self-intersection is associated to a trivial curve and a non-trivial curve, consider the innermost trivial curve $\alpha$ that is paired with a non-trivial curve via a double-curve in $A$. Since we've eliminated double-curves in $A$ associated to two trivial curves, the double-curve associated to $\alpha$ must be an innermost trivial curve on $A$. But since it is also non-trivial on $A$, it must bound a compression disk in $A$ which lifts to a compression disk in $\widetilde{A}$ bounded by the corresponding non-trivial curve on $\widetilde{A}$, contradicting essentiality of $\widetilde{A}$ in $\widetilde{N}$.
\\\\
Thus, each of the resulting $\{A_i\}$ has no self-intersections, and they are all incompressible since $\widetilde{A}$ is. It suffices to show that at least one of the $\{A_i\}$ is not boundary-parallel in $N$. Recall that a twice-punctured sphere is boundary-parallel if its closure bounds a ball containing a trivial arc of $L$. Note that each $A_i$ bounds a ball to at most one side and that each $A_i$ that bounds a ball must do so to the same side. So, if every $A_i$ bounds a ball containing a trivial arc of $L$, then $\widetilde{A}$ must bound a ball containing a trivial arc of $\widetilde{L}$ since balls in $N$ lift to balls in $\widetilde{N}$, contradicting our original choice of $\widetilde{A}$.
\\\\
Now, suppose that $N$ is non-orientable. There are three sub-cases:

\begin{enumerate}

\item $N=S\times I$ for $S$ non-orientable
\item $N=S\widetilde{\times} I$ for $S$ orientable
\item $N=S\widetilde{\times} I$ for $S$ non-orientable.

\end{enumerate}

In sub-cases (1) and (2), $N$ is double-covered by $\widetilde{N}=\widetilde{S}\times I$ for $\widetilde{S}$ orientable. In sub-case (3), $N$ is double-covered by $\widetilde{N}=\widetilde{S}\widetilde{\times} I$ for $\widetilde{N}$ orientable and $\widetilde{S}$ non-orientable. By similar arguments as in the case where $N$ is orientable, we know that the lift $\widetilde{L}$ of $L$ is prime and fully alternating in each orientable double-cover $\widetilde{N}$.
\\\\
In all cases except sub-case (3), we must have that $\widetilde{N}\setminus\widetilde{L}$ is hyperbolic by Theorem~\ref{thm:main}. By the Mostow Rigidity Theorem, we know that the deck transformation associated to the covering $\widetilde{N}\to N$ can be realized as an isometry. Thus, $N\setminus L$ must be hyperbolic as well. Finally, in sub-case (3), we reduce to the case where $N$ is orientable after taking the double cover and then apply Mostow Rigidity again to obtain the desired result.

\end{proof}

With this lemma, we are now ready to prove Theorem~\ref{thm:threemanifold}, which allows us to remove prime, fully alternating links on essential, closed surfaces in hyperbolic $3$-manifolds and preserve hyperbolicity.


\begin{proof}[Proof of Theorem~\ref{thm:threemanifold}]
We show that in the resulting manifold $M\setminus L$, we still have no essential spheres, tori or annuli.
\\\\
First note that any neighborhood $N$ of $S$ in $M$ will be an $I$-bundle over $S$. In the case that $M$ is non-orientable, we'll take an orientable double cover $\widetilde{M}$ of $M$ that lifts $N$ to its corresponding orientable double cover as given in the proof of Lemma~\ref{lem:nbd}. As in the argument at the end of the proof of Lemma~\ref{lem:nbd}, it suffices to show that $\widetilde{M}\setminus\widetilde{L}$ is hyperbolic. And, since $M$ is hyperbolic by assumption, $\widetilde{M}$ is too, so we can completely reduce to the case where $M$ is orientable.
\\\\
Now, we claim that any essential sphere, torus or annulus in $M\setminus L$ must intersect $N\setminus L$ in at least one disk or annulus. Suppose that there were an essential sphere, torus or annulus $F$ in $M\setminus (N\setminus L)$.
\\\\
Note that there are no boundary-parallel spheres in $M$ because hyperbolic manifolds do not have sphere boundaries. If $F$ is a compressible sphere in $M$, it must bound a ball $B$ in $M$. Because $S$ is incompressible in $M$, $N$ cannot be contained in $B$, so $F$ will still bound a ball in $M\setminus L$.
\\\\
If $F\subset M\setminus N$ is a torus that is boundary-parallel in $M$ but not in $M\setminus L$, then $F$ must contain a curve that bounds a once-punctured disk $D$ in $M$ that is punctured multiple times in $M\setminus L$, i.e. $L$ must puncture $D$. Since $L\subset N$, $D$ must intersect $N$ and because $S$ is incompressible in $M$, $D$ must intersect $S$ trivially, in which case we can isotope it off of $N$. But, then $D$ cannot be punctured by $L$, contradicting our assumption regarding boundary-parallelity.
\\\\
If $F\subset M\setminus N$ is a torus that has a compression disk $D$ in $M$ but not in $M\setminus L$, then $D$ must intersect $S$. Since $S$ is incompressible, we know that $D$ must intersect $S$ in a trivial curve on $S$, so we can isotope $D$ off of $N$. As before, $D$ must not be punctured by $L$, contradicting the assumption that $F$ is compressible in $M\setminus L$.
\\\\
If $F\subset M\setminus N$ is an annulus that is boundary-parallel in $M$ but not in $M\setminus L$, then $F\cup \pa M$ must contain a curve consisting of one arc from $F$ and one arc from $\pa M$ that bounds a disk $D$ punctured by $L$. Thus, $D$ must intersect $S$ and since $S$ is incompressible in $M$, $D$ must intersect $S$ trivially, so we can isotope $D$ off of $N$, leading to a similar contradiction as in the previous cases. If $F\subset M\setminus N$ is an annulus that is compressible in $M$ but incompressible in $M\setminus L$, then all compression disks of $F$ in $M$ must be punctured by $L$, so we can apply a similar argument.
\\\\
Thus, any essential sphere, torus or annulus in $M\setminus L$ must have non-empty intersection with $N$ and it follows that the intersection must contain a disk or annulus component. So, let $F$ be an essential sphere, torus or annulus with the least number of intersection curves in $\pa N$.
\\\\
If $F$ intersects $N$ in at least one disk, we can push the disk out from the boundary to decrease the number of intersection curves while preserving essentiality, which contradicts our assumption regarding minimality of intersection curves. 
\\\\
If $F\cap N$ does not contain any disks, it must contain an annulus $A$. We know that $A$ must be incompressible because $F$ is. $A$ cannot be boundary-parallel in $N$, because otherwise we could push it out through the boundary and decrease the number of intersection curves in $F\cap \pa N$. So, $A$ is an essential annulus completely contained in $N$, which contradicts Lemma~\ref{lem:nbd}.
\end{proof}

\begin{example} Here we provide an example of applying Theorem~\ref{thm:threemanifold} to generate a new hyperbolic manifold by removing a prime, fully alternating link from a hyperbolic 3-manifold with an essential closed surface. Let $M$ be the complement of the hyperbolic link shown in Figure~\ref{fig:manifoldwithsurface}. 
\begin{figure}[h]
\includegraphics[width=8cm,height=6.5cm]{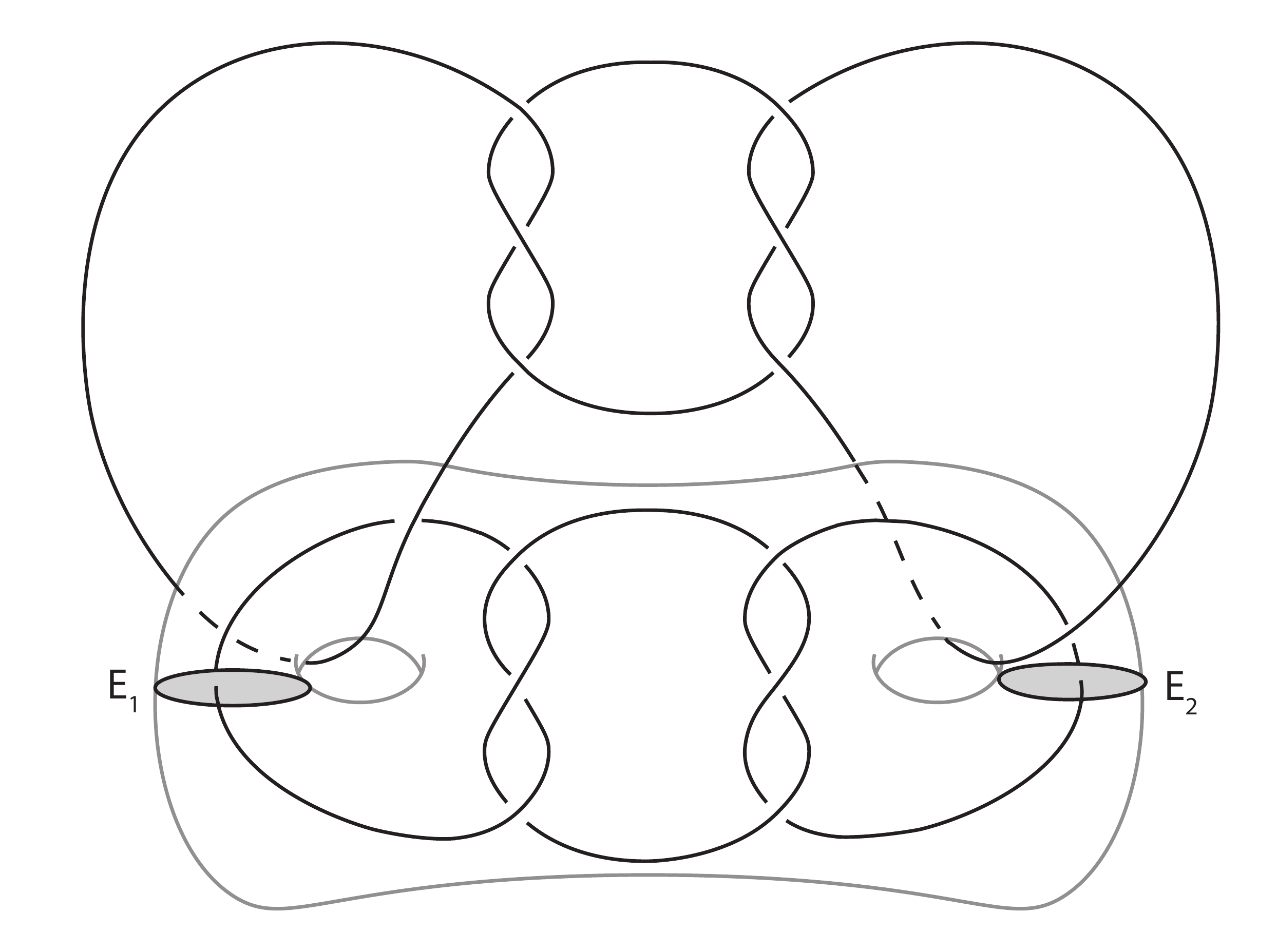}
\centering
\caption{A link complement containing an essential closed surface.}
\label{fig:manifoldwithsurface}
\end{figure}
As illustrated in Figure~\ref{fig:manifoldwithsurface}, there is a surface $S$ of genus 2 in $M$.
\\\\
We show that $S$ is incompressible. We surger $S$ along the shaded punctured disks $E_1$ amd $E_2$ in Figure~\ref{fig:manifoldwithsurface} into a four-punctured sphere $S'$. Note that $S'$ is incompressible to the inside because the tangle contained in it is not rational.
\\\\Let $D \subset M$ be a potential compression disk of $S$ to its inside such that $D$ has the fewest possible number of intersection curves with $E_1$ and $E_2$. Then each $E_i$ intersects $D$ in arcs.


If there exists an intersection arc, then there must exist an outermost arc, which together with an arc on $\pa D$ must cobound a disk $D' \subset D$ with boundary in $S'$.
 The disk $D'$ must then isotope onto $S'$, allowing us to isotope $D$ to eliminate the outermost intersection arc, contradicting the fact we chose $D$ to have a minimal number of intersection arcs.
We conclude that $S$ is incompressible to the inside. By symmetry, $S$ is also incompressible to the outside. Clearly, $S$ is not boundary-parallel, so it is essential. Now, Theorem~\ref{thm:threemanifold} implies that any prime, fully alternating link on $S$ can be removed from $M$ such that the resulting manifold is hyperbolic. For instance, the thick link in Figure~\ref{fig:addingLinkExample} is prime and fully alternating. Adding this link to the original link produces another hyperbolic link.

\begin{figure}[h]
\includegraphics[width=9.5cm,height=8cm]{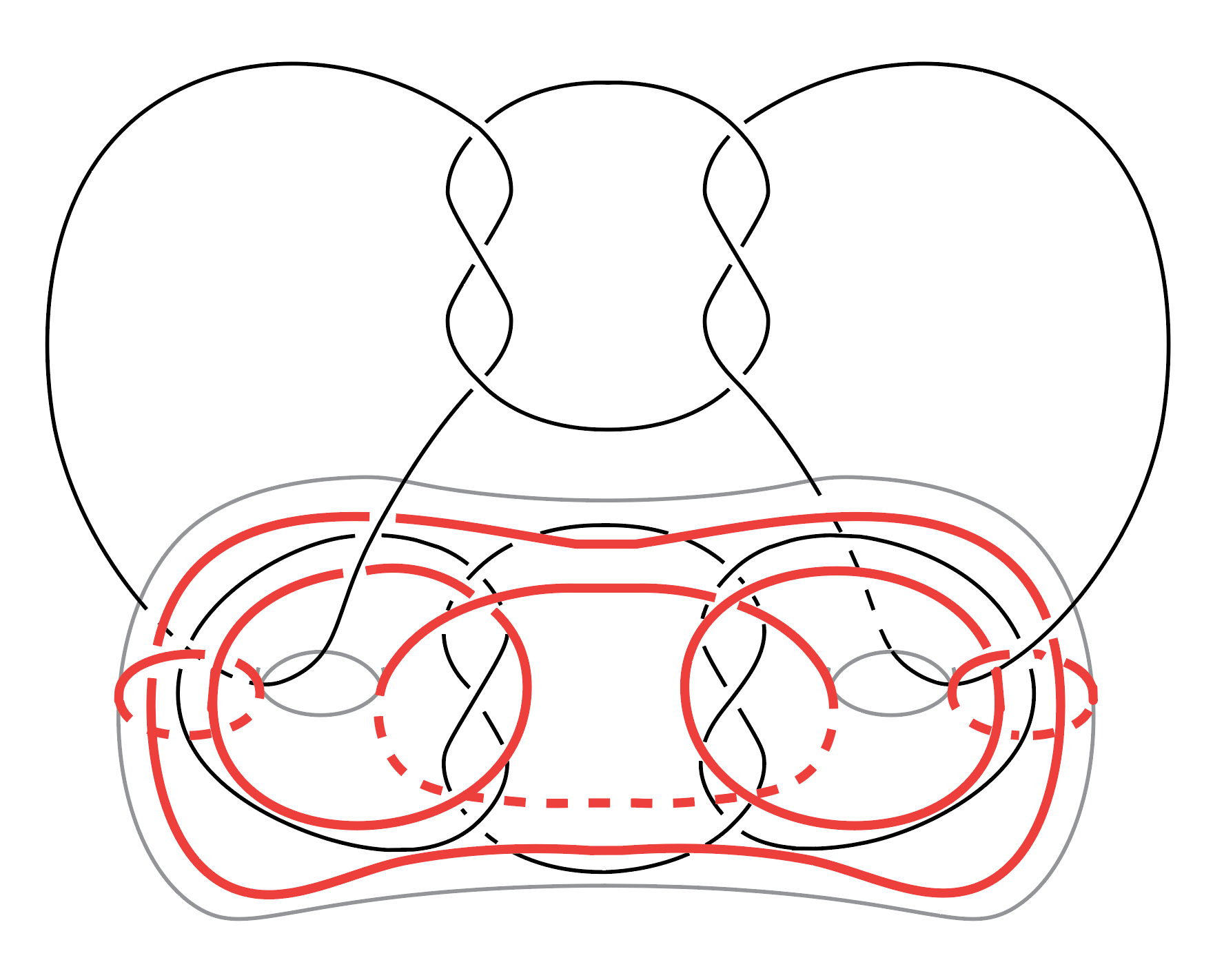}
\centering
\caption{A hyperbolic link obtained by removing a prime, fully alternating link from an essential surface on the complement of a hyperbolic link.}
\label{fig:addingLinkExample}
\end{figure}

Note that the only assumption on the tangle inside $S'$ was that it was a nontrivial non-rational tangle of two arcs. So this yields many more examples, with no requirement that the inner and outer tangles be the same.

\end{example}

\subsection*{Acknowledgments}  
The authors are grateful for support they received from NSF Grant DMS-1659037 and the Williams College SMALL REU program. 


\end{document}